\documentclass{amsart}
\usepackage{mathrsfs}
\usepackage{amsfonts}
\usepackage{amsfonts,amssymb,amsmath,amsthm}
\usepackage{url}
\usepackage{enumerate}

\newtheorem{thrm}{Theorem}[section]
\newtheorem{lem}[thrm]{Lemma}

\newtheorem{coro}[thrm]{Corollary}
\newtheorem{conj}[thrm]{Conjecture}
\theoremstyle{definition}

\newtheorem{remark}[thrm]{Remark}

\newenvironment{proof'}{{\em Proof of Main Theorem.} }{\hfill$\Box$\vspace{0.05in}}

\numberwithin{equation}{section}

\title[DEFORMING SUBMANIFOLDS]{\large{Deforming submanifolds of arbitrary codimension in a sphere}}

\author{Kefeng Liu}
\address{Center of Mathematical Sciences, Zhejiang University,
             Hangzhou,  310027, People's Republic of China; Department of Mathematics, UCLA, Box 951555, Los Angeles, CA, 90095-1555 }
\email{liu@cms.zju.edu.cn, liu@math.ucla.edu}

\author{Hongwei Xu}

\address{Center of Mathematical Sciences, Zhejiang University,
             Hangzhou,  310027, People's Republic of China}
\email{xuhw@cms.zju.edu.cn}

\author{Entao Zhao}
\address{Center of Mathematical Sciences, Zhejiang University,
             Hangzhou,  310027, People's Republic of China}
\email{zhaoet@cms.zju.edu.cn}

\thanks{Research supported by the National Natural Science Foundation
of China, Grant No. 11071211; the Trans-Century Training Programme
Foundation for Talents by the Ministry of Education of China, and
the China Postdoctoral Science Foundation, Grant No. 20090461379.}



\begin{document}

\begin{abstract} In this paper, we prove some convergence theorems for the mean curvature flow of
closed submanifolds in the unit sphere $\mathbb{S}^{n+d}$ under
integral curvature conditions. As a consequence, we obtain several
differentiable sphere theorems for certain submanifolds in
$\mathbb{S}^{n+d}$.

\end{abstract}
\maketitle

\section{Introduction}

Let $M$ be an $n$-dimensional immersed submanifold in a Riemannian
manifold $N^{n+d}$. Throughout this paper, we always assume $M$ is
connected.  Let $F_0:M\rightarrow N^{n+d}$ denote the isometric
immersion. Consider the deformation of $M$ under mean curvature
flow, i.e., consider the one-parameter family $F_{t}=F(\cdot,t)$ of
immersions $F_{t}:M\rightarrow N^{n+d}$ with corresponding images
$M_{t}=F_{t}(M)$ such that
\begin{eqnarray}\label{mcf}
\label{MCF}\left\{
\begin{array}{ll}
\frac{\partial}{\partial t}F(x,t)=H(x,t),\\
F(x,0)=F_0(x),
\end{array}\right.
\end{eqnarray}
where $H$ is the mean curvature vector of $M_t$.

The mean curvature flow was proposed by Mullins \cite{Mullins} to
describe the formation of grain boundaries in annealing metals. In
\cite{B}, Brakke introduced the motion of a submanifold in the
Euclidean space by its mean curvature in arbitrary codimension and
constructed a generalized varifold solution for all time. For the
classical solution of the mean curvature flow, most works have been
done on hypersurfaces. For the initial hypersurface satisfying
certain convexity condition, Huisken \cite{H1,H2} proved that the
solution of the mean curvature flow converges to a point as the time
approaches the finite maximal time. Later, Huisken \cite{H3}
investigated the mean curvature flow of a hypersurface in the unit
sphere $\mathbb{S}^{n+1}$. He proved that if the initial
hypersurface satisfies certain pointwise pinching condition, then
either $M_t$ converges to a round point in finite time, or $M_t$
converges to a total geodesic sphere of $\mathbb{S}^{n+1}$ as
$t\rightarrow\infty$.

For the mean curvature flow of submanifolds with higher codimension,
some important results have been obtained by Wang, Smoczyk and many
others, see \cite{Sm,Sm2,SW,WaM1,WaM2,WaM3} etc. for example.
Recently, Andrews-Baker \cite{Andrews-Baker,Baker} and
Liu-Xu-Ye-Zhao \cite{LXYZ2} proved convergence theorems for the mean
curvature flows of submanifolds satisfying certain pinching
conditions in space forms. This generalizes the convergence results
of mean curvature flow for hypersurfaces due to Huisken
\cite{H1,H2,H3} to the case of arbitrary codimensions.

An attractive question is: can one prove the convergence theorem of
the mean curvature flow of submanifolds satisfying suitable integral
curvature pinching condition? The study of convergence for the mean
curvature flow of hypersurfaces with small total curvature was
initiated in \cite{Zhao}. In \cite{LXYZ}, Liu-Xu-Ye-Zhao proved two
convergence theorems for the mean curvature flow of closed
submanifolds of arbitrary codimension in Euclidean space under
suitable integral curvature conditions.

In this paper, we investigate the convergence of mean curvature flow
of submanifolds with integral curvature bound in a sphere. In
particular, we obtain the following theorem.

\begin{thrm}\label{main-thm-1}
Let $F:M\rightarrow \mathbb{S}^{n+d}$ be an $n$-dimensional
$(n\geq3)$ smooth closed submanifold in the unit sphere with
codimension $d\geq1$. Let $F_t$ be the solution of the mean
curvature flow with $F$ as initial value. For any positive number
$p\in(n,\infty)$, there exists an explicitly computable positive constant $C_{n,p}$
depending only on $n$ and $p$, such that if
\begin{eqnarray*}
||A||_{L^p}<C_{n,p},
\end{eqnarray*}
then either $M_t$ converges to a round point in finite time, or
$M_t$ converges to a totally geodesic sphere in $\mathbb{S}^{n+d}$
as $t\rightarrow\infty$.
 \end{thrm}

In Theorem \ref{main-thm-1}, $A$ and $||\cdot||_{L^p}$ denote the
second fundamental form of a submanifold and the $L^p$-norm of a
tensor or a function, respectively.

\begin{remark}Let $M$ be a totally umbilical sphere $\mathbb{S}^{n}\Big(\frac{n}{\sqrt{|H|^2+n^2}}\Big)$ in
$\mathbb{S}^{n+d}$. Then $||A||_{L^p}={\rm
Vol}(M)^{\frac{1}{p}}\cdot |A| \leq {\rm
Vol}(\mathbb{S}^{n})^{\frac{1}{p}}\cdot |A|$. If the mean curvature
of $M$ satisfies $|H|\le\frac{\sqrt{n}C_{n,p}}{{\rm
Vol}(\mathbb{S}^{n})^{\frac{1}{p}}}$, then $||A||_{L^p}<C_{n,p}$.
Obviously, if $H=0$, $M_t$ is unchanged along the mean curvature
flow, and if $H\neq 0$, $M_t$ shrinks to a round point in finite
time. Moreover, we can construct  submanifolds from
$\mathbb{S}^{n}\Big(\frac{n}{\sqrt{|H|^2+n^2}}\Big)$ by  small
perturbations such that they satisfy $||A||_{L^p}<C_{n,p}$. If the
perturbation is small enough, then we can find the submanifold $M$
such that along the mean curvature flow $M_t$ shrinks to a round
point in finite time or converges to a totally geodesic sphere in
$\mathbb{S}^{n+d}$ as time tends to infinity.
\end{remark}

 Let $M$ be an $n$-dimensional closed submanifold immersed in a
complete simply connected $(n+d)$-dimensional space form
$\mathbb{F}^{n+d}(c)$ of constant sectional curvature $c$. The
following theorem was proved  by K. Shiohama and the second author
\cite{SX}.

\begin{thrm}[\cite{SX}]\label{shiohama-xu}Let $M$ be an $n$-dimensional $(n\geq2)$
smooth closed submanifold in $\mathbb{F}^{n+d}(c)$ with $c\geq0$.
There is an explicitly given positive constant $B_{n}$ depending
only on $n$ such that if $||\mathring{A}||_{L^n}<B_{n},$ then $M$ is
homeomorphic to a sphere.
\end{thrm}

In Theorem \ref{shiohama-xu}, $\mathring{A}$ is the traceless second
fundamental form of a submanifold.  Motivated by Theorem
\ref{shiohama-xu}, we proposed the following conjecture in
\cite{Xu-Zhao}.
\begin{conj}\label{conj}
Let $M$ be an $n$-dimensional $(n\geq2)$ smooth closed submanifold
in $F^{n+d}(c)$ with $c\geq0$. There is an positive constant $C_{n}$
depending only on $n$ such that if $||\mathring{A}||_{L^n}<C_{n},$
then $M$ is diffeomorphic to the standard $n$-sphere
$\mathbb{S}^{n}$.
\end{conj}

 As a consequence of  Theorem \ref{main-thm-1},  we have the
following differentiable sphere theorem, which can be considered as
a partial solution to Conjecture \ref{conj}.

\begin{coro}\label{Coro:3}Let $M$ be an $n$-dimensional $(n\geq3)$ smooth closed
submanifold in the unit sphere $\mathbb{S}^{n+d}$.  For any positive
number $p\in(n,\infty)$, there exists an explicitly computable positive constant $C_{n,p}$
depending only on $n$ and $p$, such that if
\begin{eqnarray*}
||A||_{L^p}<C_{n,p},
\end{eqnarray*}
then $M$ is diffeomorphic to the standard $n$-sphere
$\mathbb{S}^{n}$.
\end{coro}

At the end of this section, we would like to propose the following
conjecture which we will study by developing further the techniques
in this paper.

\begin{conj}\label{conj2}
Let $F:M\rightarrow \mathbb{R}^{n+d}$ be an $n$-dimensional
$(n\geq2)$ smooth closed submanifold in the Euclidean space with
codimension $d\geq1$. Let $F_t$ be the solution of the mean
curvature flow with $F$ as initial value. There exists a positive
constant $C_{n}$ depending only on $n$, such that if
\begin{eqnarray*}
||H||_{L^n}<n{\rm Vol}(\mathbb{S}^n)^{\frac{1}{n}}+C_{n},
\end{eqnarray*}
then $M_t$ converges to a round point in finite time. In particular,
$M$ is diffeomorphic to the standard $n$-sphere $\mathbb{S}^{n}$.
\end{conj}

When $n=2$ and $d=1$, Conjecture \ref{conj2} is closely related to
the well-known Willmore conjecture.

\section{Preliminaries}

\label{Pre} Let $F:M^{n}\rightarrow N^{n+d}$ be a smooth immersion
from an $n$-dimensional Riemannian manifold $M^n$ without boundary
to an $(n+d)$-dimensional Riemannian manifold $N^{n+d}$. We shall
make use of the following convention on the range of indices.
$$1\leq i,j,k,\cdots \leq n,\ \ 1\leq A,B,C, \cdots \leq n+d,\ \
and\ \ n+1\leq\alpha,\beta,\gamma, \cdots \leq n+d.$$

Choose a local orthonormal frame field $\{e_A\}$ in $N$ such that
$e_i$'s are tangent to $M$. Let $\{\omega_A\}$ be the dual frame
field of  $\{e_A\}$. The metric $g$ and the volume form $d\mu$ of
$M$ are $g=\sum_i \omega_i\otimes\omega_i$ and $d\mu
=\omega_1\wedge\cdots\wedge\omega_n$.

 For any $x\in M$, denoted by $N_xM$
the normal space of $M$ in $N$ at point $x$, which is  the
orthogonal complement of $T_xM$ in $F^{\ast}T_{F(x)}N$. Denote by
$\bar{\nabla}$ the Levi-Civita connection on $N$. The Riemannian
curvature tensor  $\bar{R}$  of $N$ is defined by
\begin{equation*}\bar{R}(U,V)W=-\bar{\nabla}_U\bar{\nabla}_VW+\bar{\nabla}_V\bar{\nabla}_UW+\bar{\nabla}_{[U,V]}W
\end{equation*}
for vector fields $U,V$ and $ W$ tangent to $N$. The induced
connection $\nabla$ on $M$ is defined by
\begin{equation*}\nabla_XY=(\bar{\nabla}_XY)^{\top}
\end{equation*} for $X,Y$ tangent to $M$, where $(\ )^\top$ denotes
tangential component. Let $R$ be the Riemannian curvature tensor of
$M$.

Given a normal vector field $\xi$ along $M$, the induced connection
$\nabla^\bot$ on the normal bundle is defined by
\begin{equation*}\nabla^\bot
_X\xi=(\bar{\nabla}_X\xi)^{\bot},\end{equation*} where $(\ )^{\bot}$
denotes the normal component. Let $R^\bot$ denote the normal
curvature tensor.

 The second fundamental form is defined to be
$$A(X,Y)=(\bar{\nabla}_XY)^\bot$$ as a section of the tensor bundle
$T^\ast M\otimes T^\ast M\otimes NM$, where $T^\ast M$ and $NM$ are
the cotangential bundle and  the normal bundle over $M$. The mean
curvature vector $H$ is the trace of the second fundamental form.

The first covariant derivative of $A$ is defined as
\begin{equation*}
(\widetilde{\nabla}_XA)(Y,Z)=\nabla^\bot_XA(Y,Z)-A(\nabla_XY,Z)-A(Y,\nabla_XZ),
\end{equation*}
where $\widetilde{\nabla}$ is the connection on $T^\ast M\otimes
T^\ast M\otimes NM$. Similarly, we can define the second covariant
derivative of $A$. Under the local orthonormal frame field,
 the components of $A$ and its first and second covariant
derivatives are
\begin{equation*}\begin{split}
h^\alpha_{ij}=&\langle
A(e_i,e_j), e_\alpha\rangle,\\
h^\alpha_{ijk}=&\langle (\widetilde{\nabla} _{e_k}A)(e_i,e_j),
e_\alpha\rangle,\\
h^\alpha_{ijkl}=&\langle (\widetilde{\nabla}
_{e_l}\widetilde{\nabla}_{e_k}A)(e_i,e_j), e_\alpha\rangle.
\end{split}
\end{equation*}
 Then $A$ and $H$ can be written as
\begin{eqnarray*}
A=\sum_{i,j,\alpha}h_{ij}^{\alpha}\omega_i\otimes\omega_j\otimes
e_\alpha, \ \ H=\sum_{i,\alpha}h_{ii}^{\alpha} e_\alpha=\sum_\alpha
H^{\alpha} e_\alpha.\end{eqnarray*}

The Laplacian of $A$ is defined by $\Delta
h^\alpha_{ij}=\sum_kh^{\alpha}_{ijkk}$.

We define the traceless second fundamental form $\mathring{A}$ by
$\mathring{A}=A-\frac{1}{n}g\otimes H$, whose components are
$\mathring{h}^\alpha_{ij}=h^\alpha_{ij}-\frac{1}{n}H^\alpha\delta_{ij}$.
Obviously, we have $\sum_i\mathring{h}^\alpha_{ii}=0$.

Now we recall evolution equations for some geometric quantities
associated with the evolving submanifold in the unit sphere
$\mathbb{S}^{n+d}$.
\begin{lem}[\cite{Andrews-Baker,Baker,WaM1}]\label{Lemma-evo} Along the mean curvature
flow (\ref{mcf}) where the ambient space is the unit sphere
$\mathbb{S}^{n+d}$,  we have
\begin{equation*}\begin{split}\frac{\partial}{\partial t}d\mu_t=&-H^2d\mu_t;\\
\frac{\partial}{\partial t}|A|^2=&\triangle|A|^2-2|\nabla
A|^2+2\sum_{\alpha,\beta}\Big(\sum_{i,j}h^\alpha_{ij}h^{\beta}_{ij}\Big)^2\\
&+2\sum_{i,j,\alpha,\beta}\Big[\sum_p\Big(h_{ip}^{\alpha}h_{jp}^{\beta}-h_{jp}^{\alpha}h_{ip}^{\beta}\Big)\Big]^2+4|H|^2-2n|A|^2 ;\\
\frac{\partial}{\partial t}|H|^2=&\triangle |H|^2-2|\nabla
H|^2+2\sum_{i,j}\Big(\sum_{\alpha}H^\alpha
h^{\alpha}_{ij}\Big)^2+2n|H|^2.
\end{split}
\end{equation*}
\end{lem}

From an inequality in \cite{Li-Li}, we have
\begin{equation*}
2\sum_{\alpha,\beta}\Big(\sum_{i,j}h^\alpha_{ij}h^{\beta}_{ij}\Big)^2
+2\sum_{i,j,\alpha,\beta}\Big[\sum_p\Big(h_{ip}^{\alpha}h_{jp}^{\beta}-h_{jp}^{\alpha}h_{ip}^{\beta}\Big)\Big]^2\leq
3|A|^4.
\end{equation*}
Then we have the following inequality.
\begin{equation}\label{evo-ineq-A}
\frac{\partial}{\partial t}|A|^2\leq\triangle|A|^2-2|\nabla
A|^2+3|A|^4+4|H|^2-2n|A|^2.
\end{equation}

By the Schwarz inequality, we have
\begin{equation*}\begin{split}
\sum_{i,j}\Big(\sum_{\alpha}H^\alpha h^{\alpha}_{ij}\Big)^2\leq&
\sum_{i,j}\Big(\sum_{\alpha}(H^\alpha)^2\Big)
\Big(\sum_{\alpha}(h^{\alpha}_{ij})^2\Big)\\
=&\sum_{\alpha}(H^\alpha)^2
\sum_{i,j,\alpha}(h^{\alpha}_{ij})^2=|A|^2|H|^2.
\end{split}
\end{equation*}
Hence
\begin{equation}\begin{split}\label{evo-ineq-H}
\frac{\partial}{\partial t}|H|^2
\leq&\triangle |H|^2-2|\nabla H|^2+2|A|^2|H|^2+2n|H|^2\\
=&\triangle |H|^2-2|\nabla
H|^2+2|\mathring{A}|^2|H|^2+\frac{2}{n}|H|^4+2n|H|^2.
\end{split}
\end{equation}

From Lemma \ref{Lemma-evo}, we have the evolution equation of
$|\mathring {A}|^2$.
\begin{equation}\begin{split}\label{4}
\frac{\partial}{\partial
t}|\mathring{A}|^2=&\triangle|\mathring{A}|^2-2|\nabla
\mathring{A}|^2+2\sum_{\alpha,\beta}\Big(\sum_{i,j}h^\alpha_{ij}h^{\beta}_{ij}\Big)^2\\
&+2\sum_{i,j,\alpha,\beta}\Big[\sum_p\Big(h_{ip}^{\alpha}h_{jp}^{\beta}-h_{jp}^{\alpha}h_{ip}^{\beta}\Big)\Big]^2
-\frac{2}{n}\sum_{i,j}\Big(\sum_{\alpha}H^\alpha
h^{\alpha}_{ij}\Big)^2\\
&-2n|\mathring{A}|^2.
\end{split}
\end{equation}

At the point where $H\neq0$, we choose $\{\nu_\alpha\}$ such that
$e_{n+1}=\frac{H}{|H|}$. Let
$A_H=\sum_{i,j}h_{ij}^{n+1}\omega_i\otimes\omega_j$. Set
$\mathring{A}_H=A_H-\frac{|H|}{n}{\rm Id}$  and
$|\mathring{A}_I|^2=|\mathring{A}|^2-|\mathring{A}_H|^2$. As in
\cite{Andrews-Baker}, we have the following estimate.
\begin{equation}\label{5}
\begin{split}
&2\sum_{\alpha,\beta}\Big(\sum_{i,j}h^\alpha_{ij}h^{\beta}_{ij}\Big)^2+2\sum_{i,j,\alpha,\beta}\Big[\sum_p\Big(h_{ip}^{\alpha}h_{jp}^{\beta}-h_{jp}^{\alpha}h_{ip}^{\beta}\Big)\Big]^2
-\frac{2}{n}\sum_{i,j}\Big(\sum_{\alpha}H^\alpha
h^{\alpha}_{ij}\Big)^2\\
&\leq
2|\mathring{A}_H|^4+\frac{2}{n}|\mathring{A}_H|^2|H|^2+8|\mathring{A}_H|^2|\mathring{A}_I|^2+3|\mathring{A}_I|^4\\
&\leq 2|A|^2|\mathring{A}|^2+11|\mathring{A}|^4.
\end{split}\end{equation} Combining (\ref{4}) and (\ref{5}) we obtain
\begin{equation}\label{evo-ineq-A_0}
\frac{\partial}{\partial
t}|\mathring{A}|^2\leq\triangle|\mathring{A}|^2-2|\nabla
\mathring{A}|^2+13|A|^2|\mathring{A}|^2.
\end{equation}
 Note that this inequality also holds at the point where $H=0$.

\section{A Sobolev inequality for submanifolds in a sphere}

 Firstly we recall the  well-know Michael-Simon inequality.

\begin{lem}[\cite{Michael-Simon}]\label{Michael-Simon}Let $M^n$ $(n\geq2)$ be a compact submanifold  with or without boundary
in the Euclidean space $\mathbb{R}^{n+d}$ with $d\geq1$. For a
nonnegative function  $h\in C^1(M)$ such that $h|_{\partial M}=0$ if
$\partial M\neq\emptyset$, we have
\begin{equation}\label{Michael-Simon-ineq}
\bigg(\int_Mh^{\frac{n}{n-1}}d\mu\bigg)^{\frac{n-1}{n}}\leq c_n
\int_M(|\nabla h|+|H|h)d\mu,
\end{equation}
where $c_n=4^{n+1}\sigma_n^{-1/n}$ and $\sigma_n$ is the volume of
the unit ball in $\mathbb{R}^{n}$.
\end{lem}

An improvement of the constant $c_n$ in Lemma \ref{Michael-Simon}
was given in \cite{Otsuki}. We derive a Sobolev type inequality in a
proper form, which will be used in the proof of our theorems.

\begin{lem}\label{Sobolev-ineq}
Let $M$ be an $n$-dimensional $(n\geq3)$ closed submanifold in
${\mathbb{S}}^{n + d}$. Then for all Lipschitz functions $v$ on $M$
and all $\alpha\geq \alpha_0>n$, we have
\begin{equation*}
 ||{v}||^2_{L^{\frac{2n}{n-2}}(M)}\leq C_{n,\alpha_0}\left(||{\nabla v}||^2_{L^{2}(M)} +
 \Big(1+||{H}||^{\frac{2\alpha}{\alpha-n}}_{L^\alpha(M)}\Big)||{v}||^2_{L^{2}(M)}\right),
\end{equation*}
where $C_{n,\alpha_0}$ is a positive constant depending only on $n$
and $\alpha_0$.
\end{lem}

\begin{proof}   Since a Lipschitz function is differentiable almost
everywhere, we only have to prove the lemma for  $v\in C^1(M)$ and
$v\geq0$. We consider the composition of isometric immersions $M^{n}
\rightarrow \mathbb{S}^{n+d}\subset \mathbb{R}^{n+d+1}$. Denote by
$\bar{H}$ the mean curvature vector of $M$ as a submanifold in
$\mathbb{R}^{n+d+1}$. Then we have $\bar{|H|}^2=|H|^2+n^2$. By Lemma
\ref{Michael-Simon}, we have for any nonnegative function $h\in
C^1(M)$,
\begin{equation}\label{sobo1}
\begin{split}
\bigg(\int_Mh^{\frac{n}{n-1}}d\mu\bigg)^{\frac{n-1}{n}}
\leq& c_n \int_M(|\nabla h|+|\bar{H}|h)d\mu\\
\leq& c_n \int_M(|\nabla h|+(n+|{H}|)h)d\mu.
\end{split}
\end{equation}

Let $h=v^{\frac{2(n-1)}{n-2}}$ in (\ref{sobo1}), we have

\begin{equation*}
\bigg(\int_M v^{\frac{2n}{n-2}}d\mu\bigg)^{\frac{n-1}{n}} \leq
c_n\bigg( \int_M |\nabla v|v^{\frac{n}{n-2}}d\mu+
\int_M(n+|{H}|)v^{\frac{2(n-1)}{n-2}} d\mu \bigg).
\end{equation*}

Denote by $V$ the volume of $M$ and let $Q=\frac{n}{n-2}$. By the
H\"{o}lder inequality, we have for $\alpha,\alpha'\geq \alpha_0>n$

\begin{equation*}\begin{split}
\bigg(\int_M v^{\frac{2n}{n-2}}d\mu\bigg)^{\frac{n-2}{n}}
\leq&c_n^{\frac{n-2}{n-1}}\bigg( \int_M |\nabla
v|v^{\frac{n}{n-2}}d\mu+
\int_M(n+|{H}|)v^{\frac{2(n-1)}{n-2}} d\mu \bigg)^{\frac{n-2}{n-1}}\\
\leq& \tilde{c}_n\bigg(||\nabla
v||_{L^2(M)}||v||_{L^{2Q}(M)}^{\frac{n}{n-2}}+||H||_{L^\alpha(M)}||v||_{L^{2m}(M)}^{\frac{2(n-1)}{n-2}}\\
&+V^{\frac{1}{\alpha'}}||v||_{L^{2m'}(M)}^{\frac{2(n-1)}{n-2}}\bigg)^{\frac{n-2}{n-1}}\\
 \leq& \tilde{c}_n\bigg(||\nabla
v||_{L^2(M)}^{\frac{n-2}{n-1}}
||v||_{L^{2Q}(M)}^{\frac{n}{n-1}}+||H||_{L^\alpha(M)}^{\frac{n-2}{n-1}} ||v||_{L^{2m}(M)}^2\\
&+V^{\frac{1}{\alpha'}\cdot \frac{n-2}{n-1}}||v||_{L^{2m'}(M)}^{2}
 \bigg).
 \end{split}
\end{equation*}
Here $\tilde{c}_n=n^{\frac{n-2}{n-1}}c_n^{\frac{n-2}{n-1}}$,
$m=\frac{(n-1)\alpha}{(n-2)(\alpha-1)}<\frac{n}{n-2}=Q$ and
$m'=\frac{(n-1)\alpha'}{(n-2)(\alpha'-1)}<\frac{n}{n-2}=Q$. Hence
\begin{equation}\begin{split}\label{ineq-1}
||v||_{L^{2Q}(M)}^2\leq& \tilde{c}_n\bigg(||\nabla
v||_{L^{2}(M)}^{\frac{n-2}{n-1}}||v||_{L^{2Q}(M)}^{\frac{n}{n-1}}\\
&+
||H||_{L^{\alpha}(M)}^{\frac{n-2}{n-1}}||v||_{L^{2m}(M)}^2+V^{\frac{1}{\alpha'}\cdot
\frac{n-2}{n-1} }||v||_{L^{2m'}(M)}^{2}
 \bigg).\end{split}
\end{equation}

By using Young's inequality $$ab\leq \varepsilon
a^p+\varepsilon^{-\frac{q}{p}}b^q,$$ for $a,\ b,\ \varepsilon>0$,
$p,q>1$ and $\frac{1}{p}+\frac{1}{q}=1$, with
$$a=||v||_{L^{2Q}(M)}^{\frac{n}{n-1}},\ b=||\nabla
v||_{L^{2}(M)}^{\frac{n-2}{n-1}},\
\varepsilon=\frac{1}{2\tilde{c}_n},\ p=\frac{2(n-1)}{n},\
q=\frac{2(n-1)}{n-2},$$ we obtain from (\ref{ineq-1})
\begin{equation*}\begin{split}
||v||_{L^{2Q}(M)}^2\leq&
\tilde{c}_n\bigg(\frac{1}{2\tilde{c}_n}||v||_{L^{2Q}(M)}^2+\left(\frac{1}{2\tilde{c}_n}\right)^{-\frac{n}{n-2}}||\nabla
v||_{L^{2}(M)}^2\nonumber\\
&+||H||_{L^{\alpha}(M)}^{\frac{n-2}{n-1}}||v||_{L^{2m}(M)}^2
+V^{\frac{1}{\alpha'}\cdot \frac{n-2}{n-1}
}||v||_{L^{2m'}(M)}^{2}\bigg).\end{split}
\end{equation*}
This implies
\begin{equation}\label{ineq-2}
\begin{split}
||v||_{L^{2Q}(M)}^2\leq& \hat{c}_n\bigg(||\nabla v||_{L^{2}(M)}^2\\
&+ ||H||_{L^{\alpha}(M)}^{\frac{n-2}{n-1}}||v||_{L^{2m}(M)}^2
+V^{\frac{1}{\alpha'}\cdot
\frac{n-2}{n-1}}||v||_{L^{2m'}(M)}^{2}\bigg).
\end{split}
\end{equation}
Here $\hat{c}_n=(2\tilde{c}_n)^{\frac{2(n-1)}{n-2}}$. Recall the
interpolation inequality
$$||u||_{L^r(M)}\leq \varepsilon
||u||_{L^s(M)}+\varepsilon^{-\mu}||u||_{L^t(M)},$$ where $t<r<s$,
$\mu=\frac{\frac{1}{t}-\frac{1}{r}}{\frac{1}{r}-\frac{1}{s}}$. Since
$1<m<Q$, $1<m'<Q$, we have
\begin{equation*}\begin{split}
&||v||_{L^{2m}(M)}\leq \varepsilon
||v||_{L^{2Q}(M)}+\varepsilon^{-\gamma}||v||_{L^2(M)},\\
&||v||_{L^{2m'}(M)}\leq \varepsilon'
||v||_{L^{2Q}(M)}+\varepsilon'^{-\gamma'}||v||_{L^2(M)}.
\end{split}
\end{equation*}
Here $\varepsilon,\varepsilon'>0$  to be determined,
$\gamma=\frac{Q(m-1)}{Q-m}$ and $\gamma'=\frac{Q(m'-1)}{Q-m'}$. So
we obtain from (\ref{ineq-2})
\begin{equation}\begin{split}\label{ineq-3}
||v||_{L^{2Q}(M)}^2 \leq& \hat{c}_n||\nabla
v||_{L^2(M)}^2+\hat{c}_n||H||_{L^\alpha(M)}^{\frac{n-2}{n-1}}\bigg(\varepsilon
||v||_{L^{2Q}(M)}+\varepsilon^{-\gamma}||v||_{L^2(M)}\bigg)^2\\
&+ \hat{c}_nV^{\frac{1}{\alpha'}\cdot
\frac{n-2}{n-1}}\bigg(\varepsilon'
||v||_{L^{2Q}(M)}+\varepsilon^{-\gamma'}||v||_{L^2(M)}\bigg)^2\\
\leq&\hat{c}_n||\nabla
v||_{L^2(M)}^2+\hat{c}_n||H||_{L^\alpha(M)}^{\frac{n-2}{n-1}}\bigg(\varepsilon^2
||v||_{L^{2Q}(M)}^2+\varepsilon^{-2\gamma}||v||_{L^2(M)}^2\bigg)\\
&+\hat{c}_nV^{\frac{1}{\alpha'}\cdot
\frac{n-2}{n-1}}\bigg(\varepsilon'^2
||v||_{L^{2Q}(M)}^2+\varepsilon'^{-2\gamma'}||v||_{L^2(M)}^2\bigg).
\end{split}\end{equation}
Now set
$\varepsilon^2=\frac{1}{4\hat{c}_n}||H||_{L^\alpha(M)}^{-\frac{n-2}{n-1}}$
and $\varepsilon'^2=\frac{1}{4\hat{c}_n}V^{-\frac{1}{\alpha'}\cdot
\frac{n-2}{n-1}}$. Then from (\ref{ineq-3}) we have
\begin{equation}\label{111}
\begin{split}
||v||_{L^{2Q}(M)}^2 \leq& C_n||\nabla
v||_{L^2(M)}^2+C_n^{\gamma+1}||H||_{L^\alpha(M)}^{\frac{(n-2)(1+\gamma)}{n-1}}||v||_{L^2(M)}^2\\
&+C_n^{\gamma'+1}V^{\frac{(n-2)(1+\gamma')}{(n-1)\alpha'}}||v||_{L^2(M)}^2.
\end{split}\end{equation}
Here $C_n=4\hat{c}_n\geq 1$. Notice that $C_n^{\gamma+1}$ and
$C_n^{\gamma'+1}$ are decreasing functions with respect to $\alpha$
and $\alpha'$, respectively. Then we have $C_n^{\gamma+1}\leq
C_n^{\gamma_0+1}$ and $C_n^{\gamma'+1}\leq C_n^{\gamma_0+1}$, where
$\gamma_0=\frac{n(\alpha_0+n-2)}{(n-2)(\alpha_0-n)}$ . Set
$C_{n,\alpha_0}= C_n^{\gamma_0+1}$. Letting
$\alpha'\rightarrow+\infty$, we obtain from (\ref{111})
\begin{equation}\label{112}
{\ }{\ }{\ }||v||_{L^{2Q}(M)}^2 \leq C_{n,\alpha_0}||\nabla
v||_{L^2(M)}^2+C_{n,\alpha_0}\Big(1+||H||_{L^\alpha(M)}^{\frac{(n-2)(1+\gamma)}{n-1}}\Big)||v||_{L^2(M)}^2.
\end{equation}
Substituting  $\gamma=\frac{Q(m-1)}{Q-m}$ and
$m=\frac{(n-1)\alpha}{(n-2)(\alpha-1)}$ into (\ref{112}), we obtain
\begin{equation*}
 ||{v}||^2_{L^{2Q}(M)}\leq C_{n,\alpha_0}\left(||{\nabla v}||^2_{L^{2}(M)} +
 \Big(1+||{H}||^{\frac{2\alpha}{\alpha-n}}_{L^\alpha(M)}\Big)||{v}||^2_{L^{2}(M)}\right).
\end{equation*}
This completes the proof of Lemma \ref{Sobolev-ineq}.
\end{proof}

\section{The extension of the mean curvature flow}

In this section we investigate the extension of the mean curvature
flow under finite integral curvature condition. Huisken \cite{H1,H2}
and Wang \cite{WaM1} showed that if the second fundamental form is
uniformly bounded in $[0,T)$, then the solution can be extended to
$[0,T+\omega)$ for some $\omega>0$. In
\cite{Le-Sesum,LXYZ,XYZ1,XYZ2}, the integral condition conditions
that assure the extension of the mean curvature flow were
investigated, respectively.

Now we prove another integral condition sufficiently strong to
extend the mean curvature flow. Recall that a Riemannian manifold is
said to have bounded geometry if (i) the sectional curvature is
bounded; (ii) the injective radius is bounded from below by a
positive constant.

\begin{lem}\label{lp-extension}
Let $F:M\times [0,T)\rightarrow N$ be a mean curvature flow solution
with compact initial value on a finite time interval $[0,T)$, where
$N$ has bounded geomtry. If $\int_{M_t}|A|^pd\mu_t$ is bounded in
$[0,T)$ for some $n< p<+\infty$, then the solution can be extended
to $[0,T+\omega)$ for some $\omega>0$.
\end{lem}

\begin{proof} We argue by
contradiction.

Suppose that $T(<+\infty)$ is the maximal existence time. Firstly we
choose a sequence of time $t^{(i)}$ such that $\lim_{i\rightarrow
\infty}t^{(i)}=T$. Then we take a sequence of points $x^{(i)}\in M$
satisfying
$$|A|^2(x^{(i)},t^{(i)})=\max_{(x,t)\in M\times [0,t^{(i)})}|
A|^2(x,t),$$ where $\lim_{i\rightarrow \infty}|
A|^2(x^{(i)},t^{(i)})=+\infty$.

Putting $Q^{(i)}=| A|^2(x^{(i)},t^{(i)})$, we consider the rescaling
mean curvature flow
$$F^{(i)}(t)=F\left(\frac{t}{Q^{(i)}}+t^{(i)}\right):(M, g^{(i)}(t))\rightarrow(N, Q^{(i)}h),$$
where $h$ is the metric on $N$. Then the induced metric on $M$ by
the immersion $F^{(i)}(t)$ is
$g^{(i)}(t)=Q^{(i)}g\left(\frac{t}{Q^{(i)}}+t^{(i)}\right)$, $t\in
(-Q^{(i)}t^{(i)},0]$. For $(M,g^{(i)}(t))$, the second fundamental
form $| A^{(i)}|(x,t)\leq 1$ for any $i$.

From \cite{Chen-He}, there exists a subsequence of
$(M,g^{(i)}(t),x^{(i)})$ that converges to a Riemannian manifold
$(\overline{M},\overline{g}(t),\overline{x})$, $t\in (-\infty, 0]$,
and the corresponding subsequence of immersions  $F^{(i)}(t)$
converges to an immersion $\overline{F}(t): \overline{M}\rightarrow
\mathbb{R}^{n+d}$. Then we have
\begin{equation}\label{113}
\begin{split}
\int_{B_{\overline{g}(0)}(\overline{x},1)}|
\overline{A}|_{\overline{g}(0)}^{p}d\overline{\mu}_{\overline{g}(0)}\leq&
\lim_{i\rightarrow \infty}\int_{B_{g^{(i)}(0)}(x^{(i)},1)}|
A|^{p}_{g^{(i)}(0)}d\mu_{g^{(i)}(0)}\\
=&\lim_{i\rightarrow
\infty}\frac{1}{(Q^i)^{\frac{p-n}{2}}}\cdot\int_{B_{g(t^{(i)})}(x^{(i)},(Q^{(i)})^{-\frac{1}{2}})}|
A|_{g(t^{(i)})}^{p}d\mu_{g(t^{(i)})}\\
=&0. \end{split}\end{equation}

The last equality in (\ref{113}) holds since $\lim_{t\rightarrow
T}\int_M | A |^p d\mu<+\infty$ and $Q^i\rightarrow \infty$ as
$i\rightarrow \infty$. The equality (\ref{113}) implies that
$|\overline{A}|\equiv 0$ on the ball
$B_{\overline{g}(0)}(\overline{x},1)$. In particular,
$|\overline{A}|(\overline{x},0)=0$. On the other hand, the points
selecting process implies that
$$|\overline{A}|(\overline{x},0)=\lim_{i\rightarrow \infty}|
A|_{g^{(i)}}(x^{(i)},0)=1.$$The contradiction completes the proof.
\end{proof}

\begin{remark}
Consider the totally umbilical spheres  in a complete simply
connected space form $\mathbb{F}^{n+d}(c)$ with constant curvature
$c$. Suppose the totally umbilical sphere satisfies $|H|^2+n^2c>0$.
Then along the mean curvature flow, these totally umbilical spheres
remain totally umbilical, and converge to a  round point in finite
time. On the other hand, it is easy to check that
$\int_{M_t}|A|^nd\mu_t$ is uniformly bounded along the mean
curvature flow. From these examples we see that the condition $p>n$
in Lemma \ref{lp-extension} is optimal.
\end{remark}

\section{The convergence of the mean curvature flow}
In this section, we always assume that $F_{t}$ is the solution of
the mean curvature flow of a submanifold in the unit sphere
$\mathbb{S}^{n+d}$. We first prove the following theorem.
\begin{thrm}\label{thm:convergence-S^n+d}
Let $F_0:M^n\rightarrow \mathbb{S}^{n+d}$ $(n\geq3)$ be a smooth
closed submanifold.  For given positive numbers $p\in(n,\infty)$ and
$q\in(1,\infty)$, there is a positive constant $C_1$ depending on
$n$, $p$, $q$, the upper bound  $\Lambda$ on the  $L^p$-norm of the
second fundamental form of the submanifold, such that if
\begin{equation*}||\mathring{A}||_{L^q(M_0)}<C_1,\end{equation*}
then the mean curvature flow with $F_0$ as initial value has a
unique solution $F:M\times[0,T)\rightarrow \mathbb{S}^{n+d}$, and
either

$(1)$ $T<\infty$ and $M_t$ converges to a round point as
$t\rightarrow T$; or

$(2)$ $T=\infty$ and $M_t$ converges to a totally geodesic sphere in
$\mathbb{S}^{n+d}$ as $t\rightarrow \infty$.

\end{thrm}

To prove Theorem \ref{thm:convergence-S^n+d}, we need some lemmas.

\begin{lem}\label{lem:time>T-S^n+d}If $||A||_{L^p(M_0)}\leq \Lambda$ for some $p>n$ at
$t=0$, then there is $T_1>0$ depending only on $n,p,\Lambda$ such
that there holds $||A||_{L^p(M_t)}\leq 2\Lambda$ for $t\in [0,T_1]$.
\end{lem}

\begin{proof}

Putting $u=|A|^2$, we obtain from (\ref{evo-ineq-A})
\begin{equation}\label{evo-u-S^n+p}
\frac{\partial}{\partial t}u\leq \Delta u+3u^2+2nu.
\end{equation}

From (\ref{evo-u-S^n+p}), we have
\begin{equation}\label{evo-int-u^p1111}
\begin{split}
\frac{\partial}{\partial t}\int_{M_t} u^{\frac{p}{2}}d\mu_t
=&\int_{M_t}\frac{p}{2}u^{\frac{p}{2}-1}\frac{\partial}{\partial
t}ud\mu_t
+\int_{M_t} u^{\frac{p}{2}} \frac{\partial}{\partial t}d\mu_t\\
=& \frac{p}{2}\int_{M_t} u^{\frac{p}{2}-1}(\Delta
u+c_1u^2)d\mu_t-\int_{M_t}H^2u^{\frac{p}{2}}d\mu_t\\
\leq&-\frac{4(p-2)}{p}\int_{M_t}|\nabla
u^{\frac{p}{4}}|^2d\mu_t+\frac{3p}{2}\int_{M_t}
u^{\frac{p}{2}+1}d\mu_t+np\int_{M_t} u^{\frac{p}{2}}d\mu_t.
\end{split}
\end{equation}

For the second term of the right hand side of
(\ref{evo-int-u^p1111}), we have by H\"{o}lder's inequality and
Sobolev type inequality in Lemma \ref{Sobolev-ineq},
\begin{equation}\label{ineq-fp+1-1}
\begin{split}
\int_{M_t} u^{\frac{p}{2}+1}d\mu_t \leq&
\bigg(\int_{M_t}u^{\frac{p}{2}}d\mu_t\bigg)^{{\frac{2}{p}}}\cdot
\bigg(\int_{M_t}(u^{\frac{p}{2}})^{{\frac{p}{p-2}}}d\mu_t\bigg)^{{\frac{p-2}{p}}}\\
\leq&
\bigg(\int_{M_t}u^{\frac{p}{2}}d\mu_t\bigg)^{{\frac{2}{p}}}\cdot
\bigg(\int_{M_t}u^{\frac{p}{2}}d\mu_t\bigg)^{\frac{p-n}{p}}\cdot
\bigg(\int_{M_t}(u^{\frac{p}{4}})^{\frac{2n}{n-2}}d\mu_t\bigg)^{\frac{n-2}{p}}\\
\leq&
\bigg(\int_{M_t}u^{\frac{p}{2}}d\mu_t\bigg)^{{\frac{2}{p}}}\cdot
\bigg(\int_{M_t}u^{\frac{p}{2}}d\mu_t\bigg)^{\frac{p-n}{p}}\\
&\times\bigg\{C_{n,p}\bigg(\int_{M_t}|\nabla
u^{\frac{p}{4}}|^2d\mu_t+\bigg[1+\bigg(\int_{M_t}|H|^pd\mu_t\bigg)^{\frac{2}{p-n}}\bigg]
\int_{M_t}u^\frac{p}{2}d\mu_t\bigg)\bigg\}^{\frac{n}{p}}
\end{split}
\end{equation}
\begin{equation*}
\begin{split} \leq&
\bigg(\int_{M_t}u^{\frac{p}{2}}d\mu_t\bigg)^{{\frac{p-n+2}{p}}}
\cdot\bigg[C_{n,p}^{\frac{n}{p}}\bigg(\int_{M_t}|\nabla
u^{\frac{p}{4}}|^2d\mu_t\bigg)^{\frac{n}{p}}\\
&+C_{n,p}^{\frac{n}{p}}\bigg(\int_{M_t}u^{\frac{p}{2}}d\mu_t\bigg)^{\frac{n}{p}}+n^{\frac{n}{p-n}}C_{n,p}^{\frac{n}{p}}\bigg(\int_{M_t}u^{\frac{p}{2}}d\mu_t\bigg)^{(1+\frac{2}{p-n})\frac{n}{p}}
\bigg]\\
=& C_{n,p}^{\frac{n}{p}}\bigg(\int_{M_t}u^{\frac{p}{2}}d\mu_t\bigg)^{\frac{p+2}{p}}+  n^{\frac{n}{p-n}}C_{n,p}^{\frac{n}{p}}\bigg(\int_{M_t}u^{\frac{p}{2}}d\mu_t\bigg)^{\frac{p-n+2}{p}+(1+\frac{2}{p-n})\frac{n}{p}}\\
&+C_{n,p}^{\frac{n}{p}}
\bigg(\int_{M_t}u^{\frac{p}{2}}d\mu_t\bigg)^{{\frac{p-n+2}{p}}}\cdot
\bigg(\int_{M_t}|\nabla
u^{\frac{p}{4}}|^2d\mu_t\bigg)^{\frac{n}{p}}\\
\leq& C_{n,p}^{\frac{n}{p}}\bigg(\int_{M_t}u^{\frac{p}{2}}d\mu_t\bigg)^{\frac{p+2}{p}}+n^{\frac{n}{p-n}}C_{n,p}^{\frac{n}{p}}\bigg(\int_{M_t}u^{\frac{p}{2}}d\mu_t\bigg)^{\frac{p-n+2}{p-n}}\\
&+C_{n,p}^{\frac{n}{p}}\frac{p-n}{p}\epsilon^{\frac{p}{p-n}}\bigg(\int_{M_t}u^{\frac{p}{2}}d\mu_t\bigg)^{\frac{p-n+2}{p-n}}
+C_{n,p}^{\frac{n}{p}}\frac{n}{p}\epsilon^{-\frac{p}{n}}\bigg(\int_{M_t}|\nabla
u^{\frac{p}{4}}|^2d\mu_t\bigg),
\end{split}
\end{equation*}
for any $\epsilon>0$.

Combining (\ref{evo-int-u^p1111}) and (\ref{ineq-fp+1-1}), we obtain

\begin{equation}\label{evo-int-u^p2222cc22222}
\begin{split}
\frac{\partial}{\partial t}\int_{M_t} u^{\frac{p}{2}}d\mu_t
\leq&\bigg(\frac{3n}{2}C_{n,p}^{\frac{n}{p}}\epsilon^{-\frac{p}{n}}
-\frac{4(p-2)}{p}\bigg)\int_{M_t}|\nabla
u^{\frac{p}{4}}|^2d\mu_t\\
&+np\int_{M_t}
u^{\frac{p}{2}}d\mu_t+\frac{3p}{2}C_{n,p}^{\frac{n}{p}}\bigg(\int_{M_t}u^{\frac{p}{2}}d\mu_t\bigg)^{\frac{p+2}{p}}\\
&+ \frac{3p}{2}\bigg(n^{\frac{n}{p-n}}C_{n,p}^{\frac{n}{p}}
+C_{n,p}^{\frac{n}{p}}\frac{p-n}{p}\epsilon^{\frac{p}{p-n}}\bigg)\bigg(\int_{M_t}u^{\frac{p}{2}}d\mu_t\bigg)^{\frac{p-n+2}{p-n}}
.
\end{split}
\end{equation}
Pick
$\epsilon=\Big(\frac{3npC_{n,p}^{\frac{n}{p}}}{8(p-2)}\Big)^{\frac{n}{p}}$.
Then (\ref{evo-int-u^p2222cc22222}) reduces to
\begin{equation}\label{evo-int-u^p222222222}
\frac{\partial}{\partial t}\int_{M_t} u^{\frac{p}{2}}d\mu_t \leq
np\int_{M_t}
u^{\frac{p}{2}}d\mu_t+\frac{3p}{2}C_{n,p}^{\frac{n}{p}}\bigg(\int_{M_t}u^{\frac{p}{2}}d\mu_t\bigg)^{\frac{p+2}{p}}
+c_1\bigg(\int_{M_t}u^{\frac{p}{2}}d\mu_t\bigg)^{\frac{p-n+2}{p-n}},
\end{equation}
where $c_1=\frac{3p}{2}\Big(n^{\frac{n}{p-n}}C_{n,p}^{\frac{n}{p}}
+C_{n,p}^{\frac{n}{p}}\frac{p-n}{p}\Big(\frac{3npC_{n,p}^{\frac{n}{p}}}{8(p-2)}\Big)^{\frac{n}{p-n}}\Big)$.
Then from the maximum principle and Lemma \ref{lp-extension}, there
exists a positive constant $T_1$ depending only on $n,p, \Lambda$
such that the mean curvature is smooth on $[0,T_1]$ and
$||A||_{L^p(M_t)}\leq 2\Lambda$ for $t\in [0,T_1]$. This completes
the proof of the lemma.
\end{proof}

\begin{lem}\label{lem:A^0-S^n+d}There exists a constant $T_2\in (0, T_1]$ depending only
on $n,p,q,\Lambda$ such that if
$||\mathring{A}||_{L^q(M_0)}<\varepsilon$ at $t=0$, then there holds
$||\mathring{A}||_{L^q(M_t)}\leq2\varepsilon$ for $t\in [0,T_2]$.
\end{lem}

\begin{proof}
From Lemmas \ref{Sobolev-ineq} and \ref{lem:time>T-S^n+d} we have
for a Lipschitz function $v$ and $t\in [0,T_1]$,
\begin{equation}\label{sobo2}
 ||{v}||^2_{L^{\frac{2n}{n-2}}({M_t})}\leq C_{n,p}\left(||{\nabla v}||^2_{L^{2}({M_t})} +
 \Big(1+n^{\frac{p}{p-n}}(2\Lambda)^{\frac{2p}{p-n}}\Big)||{v}||^2_{L^{2}({M_t})}\right),
\end{equation}
where $C_{n,p}$ is a positive constant depending only on $n$ and
$p$.

Define a tensor $\tilde{\mathring{A}}$ by
$\tilde{\mathring{h}}^{\alpha}_{ij}=\mathring{h}^{\alpha}_{ij}+\sigma\eta^\alpha
\delta_{ij}$, where $\eta^\alpha=1$. Set
$h_\sigma=|\tilde{\mathring{A}}|=(|\mathring{A}|^2+nd\sigma^2)^{\frac{1}{2}}$.
Then from (\ref{evo-ineq-A_0}), we have
\begin{equation}\label{evo-A^0-S^n+p22222}
\frac{\partial}{\partial t}h_\sigma\leq\triangle
h_\sigma+13|A|^2h_\sigma .
\end{equation}
For any $r\geq q>1$, we have
\begin{equation}
\begin{split}\label{ineq-h11111111}
\frac{1}{r}\frac{\partial}{\partial t}\int_{M_t} h_\sigma^{r}d\mu_t
=&\int_{M_t}h_\sigma^{r-1}\frac{\partial}{\partial t}h_\sigma d\mu_t
+\frac{1}{r}\int_{M_t} h_\sigma^{p} \frac{\partial}{\partial t}d\mu_t\\
\leq&-\frac{4(r-1)}{r^2}\int_{M_t}|\nabla h_\sigma
^{\frac{r}{2}}|^2d\mu_t+13\int_{M_t}|A|^2h_\sigma^rd\mu_t.
\end{split}
\end{equation}

For the second term of the right hand side of
(\ref{ineq-h11111111}), we have the following estimate.
\begin{equation}
\begin{split}\label{second-term0000000}
\int_{M_t}|A|^2h_\sigma^rd\mu_t
\leq&\bigg(\int_{M_t}|A|^{p}d\mu_t\bigg)^{\frac{2}{p}}\cdot\bigg(\int_{M_t}h_\sigma^{r\cdot
\frac{p}{p-2}}d\mu_t\bigg)^{\frac{p-2}{p}}\\
\leq&(2\Lambda)^{2}\bigg(\int_{M_t}h_\sigma^r
d\mu_t\bigg)^{\frac{p-n}{p}}
\cdot\bigg(\int_{M_t}(h_\sigma^r)^{\frac{n}{n-2}}
d\mu_t\bigg)^{\frac{n-2}{n}\cdot \frac{n}{p}}\\
\leq &(2\Lambda)^{2}\bigg(\int_{M_t}h_\sigma^r
d\mu_t\bigg)^{\frac{p-n}{p}}\cdot\bigg[C_{n,p}\bigg(\int_{M_t}|\nabla
h_\sigma^{\frac{r}{2}}|^2d\mu_t\\
&+\Big(1+n^{\frac{p}{p-n}}(2\Lambda)^{\frac{2p}{p-n}}
\Big)\int_{M_t}
h_\sigma^rd\mu_t\bigg)\bigg]^{\frac{n}{p}}\\
\leq&(2\Lambda)^{2}\bigg(\int_{M_t}h_\sigma^r
d\mu_t\bigg)^{\frac{p-n}{p}}\cdot\bigg[C_{n,p}^{\frac{n}{p}}\bigg(\int_{M_t}|\nabla
h_\sigma^{\frac{r}{2}}|^2d\mu_t\bigg)^{\frac{n}{p}}\\
&+C_{n,p}^{\frac{n}{p}}\Big(1+n^{\frac{p}{p-n}}(2\Lambda)^{\frac{2p}{p-n}}
\Big)^{\frac{n}{p}}\bigg(\int_{M_t}
h_\sigma^rd\mu_t\bigg)^{\frac{n}{p}}\bigg]
\end{split}
\end{equation}
\begin{equation*}
\begin{split}
=&(2\Lambda)^{2}C_{n,p}^{\frac{n}{p}}\Big(1+n^{\frac{p}{p-n}}(2\Lambda)^{\frac{2p}{p-n}}
\Big)^{\frac{n}{p}}\int_{M_t}
h_\sigma^rd\mu_t\\
&+(2\Lambda)^{2}C_{n,p}^{\frac{n}{p}}\bigg(\int_{M_t}h_\sigma^r
d\mu_t\bigg)^{\frac{p-n}{p}}\cdot\bigg(\int_{M_t}|\nabla
h_\sigma^{\frac{r}{2}}|^2d\mu_t\bigg)^{\frac{n}{p}}\\
\leq&(2\Lambda)^{2}C_{n,p}^{\frac{n}{p}}\Big(1+n^{\frac{p}{p-n}}(2\Lambda)^{\frac{2p}{p-n}}
\Big)^{\frac{n}{p}}\int_{M_t}
h_\sigma^rd\mu_t\\
&+(2\Lambda)^{2}C_{n,p}^{\frac{n}{p}}\cdot\frac{p-n}{p}\mu^{\frac{p}{p-n}}\int_{M_t}h_\sigma^r
d\mu_t\\
&+(2\Lambda)^{2}C_{n,p}^{\frac{n}{p}}\cdot\frac{n}{p}\mu^{-\frac{p}{n}}\int_{M_t}|\nabla
h_\sigma^{\frac{r}{2}}|^2d\mu_t,
\end{split}
\end{equation*}
for any $\mu>0$.

Then from (\ref{ineq-h11111111}) and (\ref{second-term0000000}) we
have
\begin{equation}
\begin{split}\label{ineq-h1ddddd1iiiii}
\frac{\partial}{\partial t}\int_{M_t} h_\sigma^{r}d\mu_t
\leq&\bigg(13r\cdot
(2\Lambda)^{2}C_{n,p}^{\frac{n}{p}}\cdot\frac{n}{p}\mu^{-\frac{p}{n}}-\frac{4(r-1)}{r}\bigg)\int_{M_t}|\nabla
h_\sigma
^{\frac{r}{2}}|^2d\mu_t\\
&+13r\cdot\bigg((2\Lambda)^{2}C_{n,p}^{\frac{n}{p}}\Big(1+n^{\frac{p}{p-n}}(2\Lambda)^{\frac{2p}{p-n}}
\Big)^{\frac{n}{p}}\\
&+(2\Lambda)^{2}C_{n,p}^{\frac{n}{p}}\cdot\frac{p-n}{p}\mu^{\frac{p}{p-n}}\bigg)\int_{M_t}h_\sigma^r
d\mu_t.
\end{split}
\end{equation}

Pick $\mu=\Big(\frac{13r^2\cdot
(2\Lambda)^{2}C_{n,p}^{\frac{n}{p}}\cdot\frac{n}{p}}{3(r-1)}\Big)^{\frac{n}{p}}$.
Then from (\ref{ineq-h1ddddd1iiiii}) we have
\begin{equation}\label{ineq-h1ddddd1}
\frac{\partial}{\partial t}\int_{M_t}
h_\sigma^{r}d\mu_t+(1-\frac{1}{q}) \int_{M_t}|\nabla h_\sigma
^{\frac{r}{2}}|^2d\mu_t\leq
c_2r^{1+\frac{n}{p-n}}\int_{M_t}h_\sigma^r d\mu_t,
\end{equation}
where
$c_2=13\bigg((2\Lambda)^{2}C_{n,p}^{\frac{n}{p}}\Big(1+n^{\frac{p}{p-n}}(2\Lambda)^{\frac{2p}{p-n}}
\Big)^{\frac{n}{p}}\cdot\frac{1}{q^{{1+\frac{n}{p-n}}}}+(2\Lambda)^{2}C_{n,p}^{\frac{n}{p}}\cdot\frac{p-n}{p}
 \Big(\frac{13q\cdot
(2\Lambda)^{2}C_{n,p}^{\frac{n}{p}}\cdot\frac{n}{p}}{3(q-1)}\Big)^{\frac{n}{p-n}}
\bigg)$.

Let $r=q$, then we have from (\ref{ineq-h1ddddd1})
\begin{equation*}
\frac{\partial}{\partial t}\int_{M_t} h_\sigma^{q}d\mu_t\leq
c_2q^{1+\frac{n}{p-n}}\int_{M_t}h_\sigma^q d\mu_t,
\end{equation*}
which implies that
\begin{equation*} \int_{M_t} h_\sigma^{q}d\mu_t\leq(2\varepsilon)^q
\end{equation*}
for $t\leq \min\{T_1,\frac{q\ln 2}{c_2q^{\frac{p}{p-n}}}\}$. Setting
$T_2=\min\{T_1,\frac{q\ln 2}{c_2q^{\frac{p}{p-n}}}\}$ and letting
$\sigma\rightarrow 0$, we complete the proof of the lemma.
\end{proof}

\begin{lem}\label{lem:A^0-pointwiseffffff}For any $t\in (0,T_2]$, we
have
\begin{equation}\label{A^0-pointwiseffffff}
|\mathring{A}|^2\leq\bigg(1+\frac{2}{n}\bigg)^{\frac{np(n+2)}{4q(p-n)}}
c_3^{\frac{n}{q}}\bigg(c_2q^{\frac{2n}{p-n}+1}+\frac{(n+2)^2}{2nt}
\bigg)^{\frac{n+2}{q}}\bigg(\int_{0}^{t}\int_{M_t}|\mathring{A}|^{q}d\mu_tdt\bigg)^{\frac{2}{q}},
\end{equation}
for some positive constants $c_2$ and $c_3$ depending only n $n,p,q$
and $\Lambda$.
\end{lem}

\begin{proof}
Fix $t_0\in (0,T_2]$. For any $\tau,\tau'$ such that
$0<\tau<\tau'<t_0$, define a function $\psi$ on $[0,t_0]$ by
\[\psi(t)=\left\{ \begin{array}{ll}
0&\ \ \ \ \ \ 0\leq t\leq \tau,\\
\frac{t-\tau}{\tau'-\tau}&\ \ \ \ \ \ \tau\leq t\leq \tau',\\
1&\ \ \ \ \ \ \tau'\leq t\leq t_0.
\end{array}
\right.\]

Then from (\ref{ineq-h1ddddd1}), we have
\begin{equation}\label{ineq-h2'111}
\frac{\partial}{\partial t}\left(\psi\int_{M_t}
f^qd\mu_t\right)d\mu_t+\Big(1-\frac{1}{p}\Big)\psi\int_{M_t}|\nabla(f^{\frac{q}{2}})|^2d\mu_t
\leq (c_2r^{\frac{p}{p-n}}\psi+\psi')\int_{M_t} f^qd\mu_t.
\end{equation}

For any $t\in[\tau', t_0]$, integrating both side of
(\ref{ineq-h2'111}) on $[\tau, t]$ implies
\begin{equation}\label{ineq-h4dddd}
\int_{M_t}
h_\sigma^{r}d\mu_t+\Big(1-\frac{1}{p}\Big)\int_{\tau'}^t\int_{M_t}|\nabla
h_\sigma ^{\frac{r}{2}}|^2d\mu_tdt \leq
\Big(c_2r^{\frac{p}{p-n}}+\frac{1}{\tau'-\tau}\Big)\int_{\tau}^{t_0}\int_{M_t}
h_\sigma^rd\mu_tdt.
\end{equation}

On the other hand, by the Sobolev inequality we have
\begin{equation}\label{ineq-h5dddd}
\begin{split}
&\int^{t_0}_{\tau'}\int_{M_t}
h_\sigma^{r(1+\frac{2}{n})}d\mu_tdt\\
\leq&\int^{t_0}_{\tau'}\bigg(\int_{M_t}
h_\sigma^{r}d\mu_t\bigg)^{\frac{2}{n}}\cdot\bigg(\int_{M_t}
h_\sigma^{\frac{nr}{n-2}}d\mu_t
\bigg)^{\frac{n-2}{n}}dt\\
\leq&\max_{t\in [\tau',t_0]}\bigg(\int_{M_t}
h_\sigma^{r}d\mu_t\bigg)^{\frac{2}{n}}\cdot\int^{t_0}_{\tau'}\bigg(\int_{M_t}
h_\sigma^{\frac{nr}{n-2}}d\mu_t \bigg)^{\frac{n-2}{n}}dt\\
\leq&C_{n,p}\cdot\max_{t\in [\tau',t_0]}\left(\int_{M_t}
h_\sigma^{r}d\mu_t\right)^{\frac{2}{n}}\int^{t_0}_{\tau'}\bigg
(\int_{M_t}|\nabla
h_\sigma^{\frac{r}{2}}|^2d\mu_t\\
&+\Big(1+n^{\frac{p}{p-n}}(2\Lambda)^{\frac{2p}{p-n}}
\Big)\int_{M_t} h_\sigma^rd\mu_t\bigg)dt.
\end{split}\end{equation}

From (\ref{ineq-h4dddd}) and (\ref{ineq-h5dddd}), we have
\begin{equation}\label{ineq-h6}
\begin{split}
\int^{t_0}_{\tau'}\int_{M_t} h_\sigma^{r(1+\frac{2}{n})}d\mu_tdt
\leq&c_3\bigg(c_2r^{\frac{2n}{p-n}+1}+\frac{1}{\tau'-\tau}\bigg)^{1+\frac{2}{n}}\\
&\times
\bigg(\int^{t_0}_{\tau}\int_{M_t}h_\sigma^rd\mu_tdt\bigg)^{1+\frac{2}{n}},
\end{split}\end{equation}
where $c_3=C_{n,p}\cdot
\max\{\frac{q}{q-1},\Big(1+n^{\frac{p}{p-n}}(2\Lambda)^{\frac{2p}{p-n}}
\Big)T_2\}$.

We put
\begin{equation*}J(r,t)=\int_t^{t_0}\int_{M_t}h_\sigma^rd\mu_tdt.
\end{equation*}
Then from (\ref{ineq-h6}) we have
\begin{equation}\label{ineq-h7}J\Big(r\Big(1+\frac{2}{n}\Big),\tau'\Big)\leq c_3\bigg(c_2r^{\frac{p}{p-n}}
+\frac{1}{\tau'-\tau}\bigg)^{1+\frac{2}{n}}J(r,\tau)^{1+\frac{2}{n}}.
\end{equation}
We let \begin{equation*}\mu=1+\frac{2}{n},\ \ r_k=q\mu^k,\ \
\tau_k=\bigg(1-\frac{1}{\mu^{k+1}}\bigg)t.\end{equation*} Notice
that $\mu>1$. From (\ref{ineq-h7}) we have
\begin{equation*}
\begin{split}J(r_{k+1},\tau_{k+1})^{\frac{1}{r_{k+1}}} \leq
c_3^{\frac{1}{r_{k+1}}}\bigg(c_2q^{\frac{p}{p-n}}
+\frac{\mu^2}{\mu-1}\cdot
\frac{1}{t}\bigg)^{\frac{1}{r_k}}\mu^{\frac{k}{r_k}\cdot
\frac{p}{p-n}}J(r_k,\tau_k)^{\frac{1}{r_{k}}}.
\end{split}\end{equation*}
Hence
\begin{equation*}
\begin{split}J(r_{m+1},\tau_{m+1})^{\frac{1}{r_{m+1}}} \leq&
c_3^{\sum_{k=0}^m\frac{1}{r_{k+1}}}\bigg(c_2q^{\frac{p}{p-n}}
+\frac{\mu^2}{\mu-1}\cdot
\frac{1}{t}\bigg)^{\sum_{k=0}^m\frac{1}{r_k}}\\
&\cdot\mu^{\frac{p}{p-n}\cdot\sum_{k=0}^m\frac{k}{r_k}}J(p,t)^{\frac{1}{p}}.
\end{split}\end{equation*}
As $m\rightarrow+\infty$, we conclude that
\begin{equation}\label{improt-estimate}
h_\sigma(x,t)\leq\bigg(1+\frac{2}{n}\bigg)^{\frac{np(n+2)}{4q(p-n)}}
c_3^{\frac{n}{2q}}\bigg(c_2q^{\frac{p}{p-n}}+\frac{(n+2)^2}{2nt}
\bigg)^{\frac{n+2}{2q}}\bigg(\int_{0}^{t_0}\int_{M_t}h_\sigma^{q}d\mu_tdt\bigg)^{\frac{1}{q}}.
\end{equation}
Now let $\sigma\rightarrow 0$. Then (\ref{improt-estimate}) implies
\begin{equation*}
|\mathring{A}|^2\leq\bigg(1+\frac{2}{n}\bigg)^{\frac{np(n+2)}{2q(p-n)}}
c_3^{\frac{n}{q}}\bigg(c_2q^{\frac{2n}{p-n}+1}+\frac{(n+2)^2}{2nt}
\bigg)^{\frac{n+2}{q}}\bigg(\int_{0}^{t_0}\int_{M_t}h_\sigma^{q}d\mu_tdt\bigg)^{\frac{2}{q}}.
\end{equation*}

Since $t_0\in (0,T_2]$ is arbitrary, we complete the proof of the
Lemma.
\end{proof}

Now we give the proof of Theorem \ref{thm:convergence-S^n+d}.
\begin{proof}[Proof of Theorem \ref{thm:convergence-S^n+d}]
We consider the submanifold $M_{T_2}$. From Lemmas
\ref{lem:A^0-S^n+d} and \ref{lem:A^0-pointwiseffffff}, we have
\begin{equation*}
|\mathring{A}|^2\leq\bigg(1+\frac{2}{n}\bigg)^{\frac{np(n+2)}{2q(p-n)}}
c_3^{\frac{n}{q}}\bigg(c_2q^{\frac{2n}{p-n}+1}+\frac{(n+2)^2}{2nT_2}
\bigg)^{\frac{n+2}{q}}T_2^\frac{2}{q}(2\varepsilon)^2:=c_4\varepsilon^2.
\end{equation*}
Set $\varepsilon_0=\Big( \frac{2}{c_4}\Big)^{\frac{1}{2}}$ for
$n\geq 4$ and $\varepsilon_0=\Big(
\frac{4}{3c_4}\Big)^{\frac{1}{2}}$ for $n=3$. If $\varepsilon\leq
\varepsilon_0$, then on $M_{T_2}$, we have $|A|^2\leq
\frac{|H|^2}{n-1}+2$ for $n\geq4$ and $|A|^2\leq
\frac{4|H|^2}{9}+\frac{4}{3}$ for $n=3$. Then by the convergence
theorem proved by Baker \cite{Baker} and the uniqueness of the mean
curvature flow, we see that the mean curvature flow with $F_0$ as
initial value either has a solution on a finite time interval
$[0,T)$ and $M_t$ converges to a round point as $t\rightarrow T$, or
has a solution on $[0,\infty)$ and $M_t$ converges to a totally
geodesic sphere in $\mathbb{S}^{n+d}$  as $t\rightarrow \infty$.
This completes the proof of Theorem \ref{thm:convergence-S^n+d}.
\end{proof}

\begin{coro}
Let $F:M^n\rightarrow \mathbb{S}^{n+d}$ $(n\geq3)$ be a smooth
closed submanifold. Let $C_1$ be as in Theorem
\ref{thm:convergence-S^n+d}. If $||{A}||_{L^p(M)}<C_1,$ then $M$ is
diffeomorphic to a unit n-sphere.
\end{coro}

Write the constant obtained in Theorem \ref{thm:convergence-S^n+d}
as  $C_1=C_1(n,p,q,\Lambda)$. Since $||\mathring{A}||_{L^p(M)}\leq
||{A}||_{L^p(M)}$, if we put $C_{n,p}=\min\{100, C_1(n,p,p,100)\}$,
then Theorem \ref{main-thm-1} follows.

\begin{thrm}\label{thm:convergence-S^n+d_H}
Let $F_0:M^n\rightarrow \mathbb{S}^{n+d}$ $(n\geq3)$ be a smooth
closed submanifold. For given positive numbers $p\in(n,\infty)$ and
$q\in(n,\infty)$, there is a positive constant $C_2$ depending on
$n$, $p$, $q$, the upper bound $\Lambda$ on the $L^p$-norm of the
mean curvature of the submanifold, such that if
\begin{equation*}||\mathring{A}||_{L^q(M_0)}<C_2,\end{equation*}
then the mean curvature flow with $F_0$ as initial value has a
unique solution $F:M\times[0,T)\rightarrow \mathbb{S}^{n+d}$, and
either

$(1)$ $T<\infty$ and $M_t$ converges to a round point as
$t\rightarrow T$; or

$(2)$ $T=\infty$ and $M_t$ converges to a totally geodesic sphere in
$\mathbb{S}^{n+d}$ as $t\rightarrow \infty$.

\end{thrm}

\begin{proof}
Suppose $||{H}||_{L^p(M_0)}\leq\Lambda$ and
$||\mathring{A}||_{L^q(M_0)}<\varepsilon$ for some fixed $p,q>n$ and
assume  $\varepsilon\in (0,100]$. Set $T=\sup\{t\in
[0,T_{\max}):||H||_{L^p(M_t)}<2\Lambda,||\mathring{A}||_{L^q(M_t)}<2\varepsilon\}.
$ We consider the mean curvature flow on the time interval [0, T).

From (\ref{evo-ineq-H}) we have for $w=|H|^2$
\begin{equation}\label{evo-w-S^n+p111}
    \frac{\partial}{\partial t}w\leq \Delta w + 2|\mathring{A}|^2w +\frac{2}{n}w^2  +
   2nw.
\end{equation}

For $r\geq \frac{p}{2}\geq \frac{3}{2}$, we have from
(\ref{evo-w-S^n+p111})
\begin{equation}\label{w-integral111dddd}
\begin{split}
\frac{1}{r}\frac{\partial}{\partial t}\int_{M_t} w^{r}d\mu_t
\leq&-\frac{4(r-1)}{r^2}\int_{M_t}|\nabla w
^{\frac{r}{2}}|^2d\mu_t\\
&+2\int_{M_t}|\mathring{A}|^2w^{r}d\mu_t+\frac{2}{n}\int_{M_t}w^{r+1}d\mu_t+{2n}\int_{M_t}w^{r}d\mu_t.
\end{split}
\end{equation}

By the definition of $T$,  we know that for any Lipschitz function
$v$ and $t\in [0,T)$, there holds
\begin{equation}\label{sobo-Mt-S^n+d}\bigg(\int_{M_t} v^{\frac{2n}{n-2}}d\mu_t\bigg)^{\frac{n-2}{n}}\leq
C_{n,p}\bigg(\int_{M_t}|\nabla
v|^2d\mu_t+\Big(1+(2\Lambda)^{\frac{2p}{p-n}}\Big)\int_{M_t}
v^2d\mu_t\bigg).
\end{equation}

For the second term of the right hand side of
(\ref{w-integral111dddd}), we have for any $\mu>0$
\begin{equation}
\begin{split}\label{second-term-H-s^n+d}
\int_{M_t}|\mathring{A}|^2w^rd\mu_t
\leq&\bigg(\int_{M_t}|\mathring{A}|^{q}d\mu_t\bigg)^{\frac{2}{q}}\cdot\bigg(\int_{M_t}w^{r\cdot
\frac{q}{q-2}}d\mu_t\bigg)^{\frac{q-2}{q}}\\
\leq&200^{2}\bigg(\int_{M_t}w^r d\mu_t\bigg)^{\frac{q-n}{q}}
\cdot\bigg(\int_{M_t}(w^r)^{\frac{n}{n-2}}
d\mu_t\bigg)^{\frac{n-2}{n}\cdot \frac{n}{q}}\\
\leq &200^{2}\bigg(\int_{M_t}w^r
d\mu_t\bigg)^{\frac{q-n}{q}}\cdot\bigg[C_{n,p}\bigg(\int_{M_t}|\nabla
w^{\frac{r}{2}}|^2d\mu_t\\
&+\Big(1+(2\Lambda)^{\frac{2p}{p-n}}\Big)\int_{M_t}
w^rd\mu_t\bigg)\bigg]^{\frac{n}{q}}\\
\leq&200^{2}\bigg(\int_{M_t}w^r
d\mu_t\bigg)^{\frac{q-n}{q}}\cdot\bigg[C_{n,p}^{\frac{n}{q}}\bigg(\int_{M_t}|\nabla
w^{\frac{r}{2}}|^2d\mu_t\bigg)^{\frac{n}{q}}\\
&+\Big(1+(2\Lambda)^{\frac{2p}{p-n}}\Big)^{\frac{n}{q}}C_{n,p}^{\frac{n}{q}}\bigg(\int_{M_t}
w^rd\mu_t\bigg)^{\frac{n}{q}}\bigg]\\
=&200^2\Big(1+(2\Lambda)^{\frac{2p}{p-n}}\Big)^{\frac{n}{q}}C_{n,p}^{\frac{n}{q}}\int_{M_t}
w^rd\mu_t\\
&+200^{2}C_{n,p}^{\frac{n}{q}}\bigg(\int_{M_t}w^r
d\mu_t\bigg)^{\frac{q-n}{q}}\cdot\bigg(\int_{M_t}|\nabla
w^{\frac{r}{2}}|^2d\mu_t\bigg)^{\frac{n}{q}}\\
\leq&200^2\Big(1+(2\Lambda)^{\frac{2p}{p-n}}\Big)^{\frac{n}{q}}C_{n,p}^{\frac{n}{q}}\int_{M_t}
w^rd\mu_t\\
&+200^{2}C_{n,p}^{\frac{n}{q}}\cdot\frac{q-n}{q}\mu^{\frac{q}{q-n}}\int_{M_t}w^r
d\mu_t+200^{2}C_{n,p}^{\frac{n}{q}}\cdot\frac{n}{q}\mu^{-\frac{q}{n}}\int_{M_t}|\nabla
w^{\frac{r}{2}}|^2d\mu_t.
\end{split}
\end{equation}

For the third term of the right hand side of
(\ref{w-integral111dddd}), we have for any $\epsilon>0$
\begin{equation}\label{last-term-H-s^n+p}
\begin{split}
\int_{M_t} w^{r+1}d\mu_t \leq&
\bigg(\int_{M_t}w^{\frac{p}{2}}d\mu_t\bigg)^{{\frac{2}{p}}}\cdot
\bigg(\int_{M_t}(w^{r})^{{\frac{p}{p-2}}}d\mu_t\bigg)^{{\frac{p-2}{p}}}\\
\leq& (2\Lambda)^2\cdot
\bigg(\int_{M_t}w^{r}d\mu_t\bigg)^{\frac{p-n}{p}}\cdot
\bigg(\int_{M_t}(w^{\frac{r}{2}})^{\frac{2n}{n-2}}d\mu_t\bigg)^{\frac{n-2}{p}}\\
\leq& (2\Lambda)^2\cdot
\bigg(\int_{M_t}w^{r}d\mu_t\bigg)^{\frac{p-n}{p}}\\
&\times\bigg[C_{n,p}\bigg(\int_{M_t}|\nabla
w^{\frac{r}{2}}|^2d\mu_t+\Big(1+(2\Lambda)^{\frac{2p}{p-n}}\Big)
\cdot\int_{M_t}w^rd\mu_t\bigg)\bigg]^{\frac{n}{p}}\\
\leq&
(2\Lambda)^2\cdot\bigg(\int_{M_t}w^{r}d\mu_t\bigg)^{{\frac{p-n}{p}}}
\cdot\bigg[C_{n,p}^{\frac{n}{p}}\bigg(\int_{M_t}|\nabla
w^{\frac{p}{4}}|^2d\mu_t\bigg)^{\frac{n}{p}}\\
&+C_{n,p}^{\frac{n}{p}}\Big(1+(2\Lambda)^{\frac{2p}{p-n}}\Big)^{\frac{n}{p}}
\bigg(\int_{M_t}w^{r}d\mu_t\bigg)^{\frac{n}{p}} \bigg]
\end{split}
\end{equation}
\begin{equation*}
\begin{split}
=&   (2\Lambda)^2\Big(1+(2\Lambda)^{\frac{2p}{p-n}}\Big)^{\frac{n}{p}}\cdot    C_{n,p}^{\frac{n}{p}}\cdot \int_{M_t}w^{r}d\mu_t\\
&+ (2\Lambda)^2\cdot C_{n,p}^{\frac{n}{p}}
\bigg(\int_{M_t}w^{r}d\mu_t\bigg)^{{\frac{p-n}{p}}}\cdot
\bigg(\int_{M_t}|\nabla
w^{\frac{p}{4}}|^2d\mu_t\bigg)^{\frac{n}{p}}\\
\leq&
(2\Lambda)^2\Big(1+(2\Lambda)^{\frac{2p}{p-n}}\Big)^{\frac{n}{p}}\cdot
C_{n,p}^{\frac{n}{p}}\cdot \int_{M_t}w^{r}d\mu_t\\
&+(2\Lambda)^2\cdot
C_{n,p}^{\frac{n}{p}}\frac{p-n}{p}\epsilon^{\frac{p}{p-n}}\cdot\int_{M_t}w^{r}d\mu_t
+(2\Lambda)^2\cdot
C_{n,p}^{\frac{n}{p}}\frac{n}{p}\epsilon^{-\frac{p}{n}}\cdot\int_{M_t}|\nabla
w^{\frac{p}{4}}|^2d\mu_t.
\end{split}
\end{equation*}

Combining (\ref{w-integral111dddd}), (\ref{second-term-H-s^n+d}) and
(\ref{last-term-H-s^n+p}) we have

\begin{equation}\label{w-integraxxxxxx}
\begin{split}
\frac{\partial}{\partial t}\int_{M_t} w^{r}d\mu_t
\leq&\bigg(2r\cdot200^{2}C_{n,p}^{\frac{n}{q}}\cdot\frac{n}{q}\mu^{-\frac{q}{n}}+\frac{2}{n}r(2\Lambda)^2\cdot
C_{n,p}^{\frac{n}{p}}\frac{n}{p}\epsilon^{-\frac{p}{n}}-\frac{4(r-1)}{r}\bigg)\int_{M_t}|\nabla
w^{\frac{r}{2}}|^2d\mu_t\\
&+\bigg(2r\cdot
200^2\Big(1+(2\Lambda)^{\frac{2p}{p-n}}\Big)^{\frac{n}{q}}C_{n,p}^{\frac{n}{q}}+2r\cdot200^{2}C_{n,p}^{\frac{n}{q}}\cdot\frac{q-n}{q}\mu^{\frac{q}{q-n}}\\
&+\frac{2}{n}r\cdot
(2\Lambda)^2\Big(1+(2\Lambda)^{\frac{2p}{p-n}}\Big)^{\frac{n}{p}}\cdot
C_{n,p}^{\frac{n}{p}}+ \frac{2}{n}r\cdot(2\Lambda)^2\cdot
C_{n,p}^{\frac{n}{p}}\frac{p-n}{p}\epsilon^{\frac{p}{p-n}}\\
&+2nr\bigg)\int_{M_t}w^{r}d\mu_t.
\end{split}
\end{equation}

 Set
$c_5=2\cdot200^{2}C_{n,p}^{\frac{n}{q}}\cdot\frac{n}{q}+\frac{2}{n}(2\Lambda)^2\cdot
C_{n,p}^{\frac{n}{p}}\frac{n}{p}$ and $c_6=2\cdot
200^2\Big(1+(2\Lambda)^{\frac{2p}{p-n}}\Big)^{\frac{n}{q}}C_{n,p}^{\frac{n}{q}}
+2\cdot200^{2}C_{n,p}^{\frac{n}{q}}\cdot\frac{q-n}{q}\Big(\frac{c_5p}{4(p-2)}\Big)^{\frac{n}{q-n}}+\frac{2}{n}\cdot
(2\Lambda)^2\Big(1+(2\Lambda)^{\frac{2p}{p-n}}\Big)^{\frac{n}{p}}\cdot
C_{n,p}^{\frac{n}{p}}+ \frac{2}{n}\cdot(2\Lambda)^2\cdot
C_{n,p}^{\frac{n}{p}}\frac{p-n}{p}\Big(\frac{c_5p}{4(p-2)}\Big)^{\frac{n}{p-n}}+2n$.
Let $\mu=\Big(\frac{c_5r^2}{4(r-1)}\Big)^{\frac{n}{q}}$ and
$\epsilon=\Big(\frac{c_5r^2}{4(r-1)}\Big)^{\frac{n}{p}}$. Then from
(\ref{w-integraxxxxxx}) we get
\begin{equation}\label{w-integrfffffdddf}
\frac{\partial}{\partial t}\int_{M_t} w^{r}d\mu_t\leq c_6
r^{\max\{\frac{n}{p-n},\frac{n}{q-n}\}+1}\int_{M_t}w^{r}d\mu_t.
\end{equation}
Take $r=\frac{p}{2}$. Then for $t\in[0,\min\{T,T_1\})$, where
$T_1=\frac{p\ln \frac{3}{2}}{ c_6
(\frac{p}{2})^{\max\{\frac{n}{p-n},\frac{n}{q-n}\}+1}}$, there holds
$||H||_{L^p(M_t)}<\frac{3}{2}\Lambda$.

From (\ref{evo-ineq-A_0}) we have the following inequality for
$h_\sigma$.
\begin{equation}\label{h-inequality}
\frac{\partial}{\partial t}h_\sigma\leq\triangle
h_\sigma+13|\mathring{A}|^2h_\sigma+\frac{2}{n}|H|^2h_\sigma.
\end{equation}

For any $r\geq q>1$, we have
\begin{equation}
\begin{split}\label{ineq-h1111111ddddddd1}
\frac{1}{r}\frac{\partial}{\partial t}\int_{M_t} h_\sigma^{r}d\mu_t
=&\int_{M_t}h_\sigma^{r-1}\frac{\partial}{\partial t}h_\sigma d\mu_t
+\frac{1}{r}\int_{M_t} h_\sigma^{p} \frac{\partial}{\partial t}d\mu_t\\
\leq&-\frac{4(r-1)}{r^2}\int_{M_t}|\nabla h_\sigma
^{\frac{r}{2}}|^2d\mu_t+13\int_{M_t}|\mathring{A}|^2h_\sigma^rd\mu_t+\frac{2}{n}\int_{M_t}|H|^2h_\sigma^rd\mu_t.
\end{split}
\end{equation}

For the second term of the right hand side of
(\ref{ineq-h1111111ddddddd1}), as (\ref{second-term-H-s^n+d}) we
have for any $\nu>0$
\begin{equation}
\begin{split}\label{second-term-H1122uu22}
\int_{M_t}|\mathring{A}|^2h_\sigma^rd\mu_t
\leq&200^2\Big(1+(2\Lambda)^{\frac{2p}{p-n}}\Big)^{\frac{n}{q}}C_{n,p}^{\frac{n}{q}}\int_{M_t}
h_\sigma^rd\mu_t\\
&+200^{2}C_{n,p}^{\frac{n}{q}}\cdot\frac{q-n}{q}\nu^{\frac{q}{q-n}}\int_{M_t}h_\sigma^r
d\mu_t\\
&+200^{2}C_{n,p}^{\frac{n}{q}}\cdot\frac{n}{q}\nu^{-\frac{q}{n}}\int_{M_t}|\nabla
h_\sigma^{\frac{r}{2}}|^2d\mu_t.
\end{split}
\end{equation}

Similarly, for the last term of the right hand side of
(\ref{ineq-h1111111ddddddd1}), we have for any $\vartheta >0$

\begin{equation}
\begin{split}\label{third-term-H2233332}
\int_{M_t}|H|^2h_\sigma^rd\mu_t
\leq&(2\Lambda)^2\Big(1+(2\Lambda)^{\frac{2p}{p-n}}\Big)^{\frac{n}{p}}C_{n,p}^{\frac{n}{p}}\int_{M_t}
h_\sigma^rd\mu_t\\
&+(2\Lambda)^{2}C_{n,p}^{\frac{n}{p}}\cdot\frac{p-n}{p}\vartheta^{\frac{p}{p-n}}\int_{M_t}h_\sigma^r
d\mu_t\\
&+(2\Lambda)^{2}C_{n,p}^{\frac{n}{p}}\cdot\frac{n}{p}\vartheta
^{-\frac{p}{n}}\int_{M_t}|\nabla h_\sigma^{\frac{r}{2}}|^2d\mu_t.
\end{split}
\end{equation}

Combining (\ref{ineq-h1111111ddddddd1}),
(\ref{second-term-H1122uu22}) and (\ref{third-term-H2233332}), we
obtain
\begin{equation}
\begin{split}\label{ineq-hhhhheee}
&\frac{\partial}{\partial t}\int_{M_t} h_\sigma^{r}d\mu_t\\
\leq&\bigg(13r\cdot200^{2}C_{n,p}^{\frac{n}{q}}\cdot\frac{n}{q}\nu^{-\frac{q}{n}}+\frac{2}{n}r\cdot
(2\Lambda)^{2}C_{n,p}^{\frac{n}{p}}\cdot\frac{n}{p}\vartheta
^{-\frac{p}{n}} -\frac{4(r-1)}{r}\bigg)\int_{M_t}|\nabla h_\sigma
^{\frac{r}{2}}|^2d\mu_t\\
&+\bigg(13r\cdot200^2\Big(1+(2\Lambda)^{\frac{2p}{p-n}}\Big)^{\frac{n}{q}}C_{n,p}^{\frac{n}{q}}
+13r\cdot200^{2}C_{n,p}^{\frac{n}{q}}\cdot\frac{q-n}{q}\nu^{\frac{q}{q-n}}\\
&+\frac{2}{n}r\cdot
(2\Lambda)^2\Big(1+(2\Lambda)^{\frac{2p}{p-n}}\Big)^{\frac{n}{p}}C_{n,p}^{\frac{n}{p}}
+\frac{2}{n}r\cdot(2\Lambda)^{2}C_{n,p}^{\frac{n}{p}}\cdot\frac{p-n}{p}\vartheta^{\frac{p}{p-n}}
 \bigg)\int_{M_t}
h_\sigma^rd\mu_t.
\end{split}
\end{equation}

Set
$c_7=13\cdot200^{2}C_{n,p}^{\frac{n}{q}}\cdot\frac{n}{q}+\frac{2}{n}\cdot
(2\Lambda)^{2}C_{n,p}^{\frac{n}{p}}\cdot\frac{n}{p}$ and
$c_8=13\cdot200^2\Big(1+(2\Lambda)^{\frac{2p}{p-n}}\Big)^{\frac{n}{q}}C_{n,p}^{\frac{n}{q}}
+13\cdot200^{2}C_{n,p}^{\frac{n}{q}}\cdot\frac{q-n}{q}\Big(\frac{c_7q}{3(q-1)}\Big)^{\frac{n}{q-n}}+\frac{2}{n}\cdot
(2\Lambda)^2\Big(1+(2\Lambda)^{\frac{2p}{p-n}}\Big)^{\frac{n}{p}}C_{n,p}^{\frac{n}{p}}
+\frac{2}{n}\cdot(2\Lambda)^{2}C_{n,p}^{\frac{n}{p}}\cdot\frac{p-n}{p}\Big(\frac{c_7q}{3(q-1)}\Big)^{\frac{n}{p-n}}$.
Let $\nu=\Big(\frac{c_7r^2}{3(r-1)}\Big)^{\frac{n}{q}}$ and
$\vartheta=\Big(\frac{c_7r^2}{3(r-1)}\Big)^{\frac{n}{p}}$. Then
(\ref{ineq-hhhhheee}) implies
\begin{equation}
\begin{split}\label{ineq-ffffffee}
\frac{\partial}{\partial t}\int_{M_t} h_\sigma^{r}d\mu_t+
\bigg(1-\frac{1}{q}\bigg)\int_{M_t}|\nabla h_\sigma
^{\frac{r}{2}}|^2d\mu_t\leq
c_8r^{\max\{\frac{n}{p-n},\frac{n}{q-n}\}+1} \int_{M_t}
h_\sigma^rd\mu_t.
\end{split}
\end{equation}
Take $r=p$. Then for $t\in [0,\min\{T,T_2\})$, where $T_2=
\frac{p\ln \frac{3}{2}}{ c_8
q^{\max\{\frac{n}{p-n},\frac{n}{q-n}\}+1}}$, we have
$||\mathring{A}||_{L^q(M_t)}<\frac{3}{2}\varepsilon$.

We claim that $T>\min\{T_1,T_2\}$. Suppose not, i.e.,
$T\leq\min\{T_1,T_2\}$. If $T<T_{\max}$, then by the smooth of the
mean curvature flow and the definition of $T$, we get a
contradiction. If $T=T_{\max}$, then $T_{\max}$ must be $\infty$. If
not, by the definition of $T$, for $t\in [0,T_{\max})$ we have
$||H||_{L^p(M_t)}<2\Lambda$ and
$||\mathring{A}||_{L^q(M_t)}<2\varepsilon\leq 200$. This implies
$||A||_{L^{\min\{p,q\}}(M_t)}<\infty$ for any $t\in [0,T_{\max})$.
Then by Lemma \ref{lp-extension}, the mean curvature flow can be
extended over time $T_{\max}$, which is a contradiction. Hence we
obtain that $T>\min\{T_1,T_2\}$.

Set $T_0=\min\{T_1,T_2\}$. We consider the mean curvature flow on
$[0,\frac{T_0}{2}]$. Then we know that (\ref{ineq-ffffffee}) holds
for any $t\in [0,\frac{T_0}{2}]$.
 By a standard Moser iteration as before, we have for  any $t\in
 (0,\frac{T_0}{2}]$, there holds
\begin{equation*}
h_\sigma(x,t)\leq\bigg(1+\frac{2}{n}\bigg)^{\hat{q}}
c_{9}^{\frac{n}{2q}}\bigg(c_{8}q^{\max\{\frac{n}{p-n},\frac{n}{q-n}\}
+1}+\frac{(n+2)^2}{2nt}
\bigg)^{\frac{n+2}{2q}}\bigg(\int_{0}^{\frac{T_0}{2}}\int_{M_t}h_\sigma^{q}d\mu_tdt\bigg)^{\frac{1}{q}},
\end{equation*}
where  $c_{9}=C_{n,p}\cdot \max\Big\{\frac{q}{q-1},
\Big(1+(2\Lambda)^{\frac{2p}{p-n}}\Big)\cdot T_0\Big\}$ and
$\hat{q}=\frac{n(n+2)}{4q}\cdot\Big(\max\{\frac{n}{p-n},\frac{n}{q-n}\}
+1 \Big)$. Letting  $\sigma\rightarrow 0$, we get at time
$\frac{T_0}{2}$
\begin{equation*}
|\mathring{A}|^2\leq\bigg(1+\frac{2}{n}\bigg)^{2\hat{q}}
c_{9}^{\frac{n}{q}}\bigg(c_{26}q^{\max\{\frac{n}{p-n},\frac{n}{q-n}\}
+1}+\frac{(n+2)^2}{nT_0} \bigg)^{\frac{n+2}{q}}\cdot
\Big(\frac{T_0}{2}\Big)^{\frac{2}{q}}(2\varepsilon)^2:=c_{10}\varepsilon^2.
\end{equation*}

Set $\varepsilon_0=\Big( \frac{2}{c_{10}}\Big)^{\frac{1}{2}}$ for
$n\geq 4$ and $\varepsilon_0=\Big(
\frac{4}{3c_{10}}\Big)^{\frac{1}{2}}$ for $n=3$. Then if
$\varepsilon\leq \varepsilon_0$, we have $|A|^2\leq
\frac{|H|^2}{n-1}+2$ for $n\geq4$ and $|A|^2\leq
\frac{4|H|^2}{9}+\frac{4}{3}$ for $n=3$ on $M_{\frac{T_0}{2}}$. By
the convergence theorem proved by Baker \cite{Baker} and the
uniqueness of the mean curvature flow, we see that the mean
curvature flow with $F_0$ as initial value either has a solution on
a finite time interval $[0,T)$ and $M_t$ converges to a round point
as $t\rightarrow T$, or has a solution on $[0,\infty)$ and $M_t$
converges to a totally geodesic sphere in $\mathbb{S}^{n+d}$  as
$t\rightarrow \infty$. This completes the proof of Theorem
\ref{thm:convergence-S^n+d_H}.
\end{proof}

\begin{coro}
Let $F:M^n\rightarrow \mathbb{S}^{n+d}$ $(n\geq3)$ be a smooth
closed submanifold. Let $C_2$ be as in Theorem
\ref{thm:convergence-S^n+d_H}. If $||{A}||_{L^p(M)}<C_2,$ then $M$
is diffeomorphic to a unit n-sphere.
\end{coro}

Write the constant obtained in Theorem \ref{thm:convergence-S^n+d_H}
as  $C_2=C_2(n,p,q,\Lambda)$. Since $||\mathring{A}||_{L^p(M)}\leq
||{A}||_{L^p(M)}$, if we put $C_{n,p}=\min\{100, C_2(n,p,p,100\cdot
n^{\frac{1}{2}})\}$, then Theorem \ref{main-thm-1} also follows.
Hence, the pinching constant $C_{n,p}$ in Theorem \ref{main-thm-1}
can be chosen as  $C_{n,p}=\min \Big\{100,
\max\{C_1(n,p,p,100),C_2(n,p,p,100\cdot n^{\frac{1}{2}})\}\Big\}$.

\begin{remark} From the proofs of Lemma \ref{Sobolev-ineq}, Theorem
\ref{thm:convergence-S^n+d} and Theorem
\ref{thm:convergence-S^n+d_H}, the constant $C_{n,p}$ in Theorem
\ref{main-thm-1} can be computed explicitly.
\end{remark}

\end{document}